\documentclass[11pt]{amsart}
\usepackage{amssymb,amsfonts,amsmath,amscd,wasysym,amsthm}
\usepackage{verbatim} 
\usepackage[all]{xy} 
\usepackage{hyperref}

\theoremstyle{plain}

\newtheorem{theorem}{Theorem}[section]

\newtheorem{corollary}[theorem]{Corollary}
\newtheorem{proposition}[theorem]{Proposition}

\newtheorem*{theorem*}{Theorem}
\newtheorem*{corollary*}{Corollary}

\theoremstyle{definition} 
\newtheorem{definition}[theorem]{Definition}

\newtheorem{remark}[theorem]{Remark}
\newtheorem{closingremark}[theorem]{Closing Remark}
\newtheorem{example}[theorem]{Example}

\newcommand{\fm}{\mbox{$\mathfrak{m}$}}
\newcommand{\fn}{\mbox{$\mathfrak{n}$}}
\newcommand{\fp}{\mbox{$\mathfrak{p}$}}

\newcommand{\fV}{\mbox{$\mathfrak{V}$}}
\newcommand{\fI}{\mbox{$\mathfrak{I}$}}

\newcommand{\bn}{\mbox{$\mathbb{N}$}}
\newcommand{\bz}{\mbox{$\mathbb{Z}$}}

\newcommand{\bc}{\mbox{$\mathbb{C}$}}
\newcommand{\ba}{\mbox{$\mathbb{A}$}}

\newcommand{\mcg}{\mbox{$\mathcal{G}$}}

\newcommand{\height}{\mbox{${\rm height}$}}
\newcommand{\rt}{\mbox{${\rm rt}$}} 
\newcommand{\Spec}{\mbox{${\rm Spec}$}} 
\newcommand{\Max}{\mbox{${\rm Max}$}}

\newcommand{\rees}{\mbox{${\bf R}$}}
\newcommand{\sym}{\mbox{${\bf S}$}}
\newcommand{\agr}{\mbox{${\bf G}$}}

\title{The relation type of affine algebras and algebraic
  varieties}

\author{{\sc  Francesc Planas-Vilanova }}

\date{\today}


\begin{document}

\begin{abstract}
We introduce the notion of relation type of an affine algebra and
prove that it is well defined by using the Jacobi-Zariski exact
sequence of Andr\'e-Quillen homology. In particular, the relation type
is an invariant of an affine algebraic variety. Also as a consequence
of the invariance, we show that in order to calculate the relation
type of an ideal in a polynomial ring one can reduce the problem to
trinomial ideals. When the relation type is at least two, the extreme
equidimensional components play no role. This leads to the non
existence of affine algebras of embedding dimension three and relation
type two.
\end{abstract}

\maketitle

\section{Introduction}\label{introduction}

Let $A=R/I=k[x_1,\ldots ,x_n]/I$ be an affine $k$-algebra, where $k$
is a field, $x_1,\ldots ,x_n$ are variables over $k$ and
$I=(f_1,\ldots ,f_s)$ is an ideal of the polynomial ring
$R=k[x_1,\ldots ,x_n]$.

In this note we introduce the following invariant of $A$: the {\em
  relation type of} $A$ is defined as $\rt(A)=\rt(I)$, where $\rt(I)$
stands for the relation type of the ideal $I$. 

Recall that if $\rees(I)=R[It]=\oplus_{q\geq 0}I^{q}t^q$ is the {\em
  Rees algebra} of $I$ and $\varphi:S=R[t_1,\ldots ,t_s]\to\rees(I)$
is the natural graded polynomial presentation sending $t_i$ to $f_it$,
then $L=\ker(\varphi)=\oplus_{q\geq 1}L_q$, referred to as the ideal
of {\em equations} of $I$, is a graded ideal of $S$ and the relation
type of $I$, denoted by $\rt(I)$, is the least integer $N\geq 1$ such
that $L$ is generated by its components of degree at most
$N$. Concerning the equations of an ideal and its relation type, see
for instance, and with no pretense of being exhaustive, \cite{hku},
\cite{hmv}, \cite{hsv}, \cite{morey}, \cite{mu}, \cite{mp},
\cite{planas}, \cite{tchernev}, \cite{vasconcelos-Crelle},
\cite{vasconcelos-Comput}, \cite{vasconcelos-Integral}, and the
references therein. 

By means of the Jacobi-Zariski exact sequence of Andr\'e-Quillen
homology, we prove that the definition of $\rt(A)$ does not depend on
the presentation of $A$. In particular, we obtain an invariant for
affine algebraic varieties. Concretely, if $V$ is an affine algebraic
$k$-variety, the {\em relation type} of $V$ is defined as
$\rt(V)=\rt(k[V])$, the relation type of its coordinate ring $k[V]$.
Another consequence is that in order to calculate the relation type of
an ideal in a polynomial ring one can reduce the problem to ideals
generated by trinomials, though at the cost of introducing more
generators and more variables.

We then study the connection between the equidimensional decomposition
of a radical ideal $I$ of $R=k[x_1,\ldots ,x_n]$ and its relation type
$\rt(I)$. We conclude that the equidimensional components of dimension
$1$ and $n$ are not relevant whenever the relation type is at least
two. As a corollary we obtain a somewhat surprising result, namely,
that there are no affine $k$-algebras of embedding dimension three and
relation type two.

Due to the aforementioned result, the examples we provide are
essentially focussed on affine space curves. At this point one should
emphasize that the explicit calculation of the equations of an ideal
is computationally a very expensive task. The reduction to trinomial
ideals, unfortunately, does not seem to improve, in general, the
approach to the problem. It would be desirable to obtain a wide range
of irreducible affine space curves with prescribed relation type. It
would also be interesting to understand better the geometric meaning
of the relation type of an algebraic variety.

Notice that, to our knowledge, there is at least another notion also
named {\em relation type of an algebra}. Indeed, in
\cite[Definition~2.7]{vasconcelos-Six}, W.V. Vasconcelos defines the
relation type of a standard algebra $A=k[x_1,\ldots ,x_n]/I$ as the
least integer $s$ such that $I=(I_1,\ldots ,I_s)$, where $I\subset
(x_1,\ldots ,x_n)^2$ is a homogeneous ideal of the polynomial ring
$k[x_1,\ldots ,x_n]$ over the field $k$ and $I_i$ is the $i$-th graded
component of $I$. It is clear that both definitions do not coincide,
even in the homogeneous case; for instance, take $I=(f)$ a principal
ideal of $k[x_1,\ldots ,x_n]$ generated by a homogeneous polynomial
$f$ of degree $p\geq 2$.

The paper is organized as follows. In Section~\ref{notations}, we set
the notations used throughout. In Section~\ref{invariance}, we prove
the invariance theorem and the reduction to
trinomials. Section~\ref{equidimensionalvsrt} is devoted to study the
effect of an equidimensional decomposition in the computation of the
relation type. Finally, in Sections~\ref{examples} and
\ref{spacecurves}, we give some examples in embedding dimension three.

\section{Notations and preliminaries}\label{notations}

We begin by setting some notations. Let $I=(f_1,\ldots ,f_s)$ be an
ideal of a Noetherian ring $R$ and let $\rees(I)=R[It]=\oplus_{q\geq
  0}I^{q}t^q$ be its Rees ring. Let $S=R[t_1,\ldots ,t_s]$ be the
polynomial ring over $R$ and $\varphi:S\to\rees(I)$ the graded
polynomial presentation of $\rees(I)$ sending $t_i$ to $f_it$. Set
$L=\ker(\varphi)=\oplus_{q\geq 1}L_q$, the graded ideal of equations
of $I$. Given $q\geq 1$, let $L\langle q\rangle \subseteq L$ be the
ideal generated by the homogeneous equations of $I$ of degree at most
$q$. The relation type of $I$, denoted by $\rt(I)$, is the least
integer $N\geq 1$ such that $L=L\langle N\rangle$. Although $L$
depends on the chosen generating set, $\rt(I)$ does not. Indeed, let
$\sym(I)$ be the symmetric algebra of $I$ and let
$\alpha:\sym(I)\to\rees(I)$ be the canonical graded morphism induced
by the identity on degree $1$. Let $0\to Z_1\to R^s\to I\to 0$ be the
presentation associated to $f_1,\ldots ,f_s$, where $Z_1$ stands for
the first module of syzygies of $I$. Applying the symmetric functor,
one gets the graded exact sequence
\begin{eqnarray*}
0\rightarrow Z_1S\longrightarrow S=R[t_1,\ldots
  ,t_s]=\sym(R^s)\xrightarrow{\Phi}\sym(I)\rightarrow 0,
\end{eqnarray*}
where $Z_1S$ is the ideal of $S$ generated by the elements of $Z_1$
regarded as linear forms of $S$. Thus $\varphi=\alpha\circ\Phi:S\to
\rees(I)$ and $Z_1S$ can be interpreted as the ideal $L\langle
1\rangle$ of linear equations of $I$. For each $q\geq 2$, set
$E(I)_q=\ker(\alpha_q)/I\cdot\ker(\alpha_{q-1})$ and call it the
module of {\em effective $q$-relations} of $I$. One can show that the
so-called module of {\em fresh generators in degree} $q$, $(L/L\langle
q-1\rangle)_q=L_q/S_1L_{q-1}$, is isomorphic to $E(I)_q$ (see, e.g.,
\cite[Before Definition~1.9]{vasconcelos-Integral} and
\cite[Theorem~2.4]{planas}). In particular, $L_q/S_1L_{q-1}$ does not
depend on the presentation of $I$. Furthermore, $\rt(I)$ can be
thought as $\rt(I)=\min\{ r\geq 1\mid E(I)_q=0\mbox{ for all }q\geq
r+1\}$.  It follows that $\rt(I)$ is independent too of the
presentation of $I$ and that of $\rees(I)$.

Observe that $\rt(I)=1$ if and only if $\alpha:\sym(I)\to \rees(I)$ is
an isomorphism. In such a case $I$ is said to be of {\em linear
  type}. If $\alpha_2:\sym_2(I)\to I^2$ is an isomorphism, $I$ is said
to be {\em syzygetic}. It is well-known that ideals generated by a
regular sequence (more in general, by a $d$-sequence) are of linear
type (see, e.g., \cite{hsv} or \cite{hmv}). We recall too that the
relation type is a local invariant, in other words,
$\rt(I)=\sup\{\rt(I_{\mathfrak{p}})\mid
\fp\in\Spec(A)\}=\sup\{\rt(I_{\mathfrak{m}})\mid \fm\in\Max(A)\}$
(see, e.g., \cite[Example~3.2]{planas}; here $\Max$ stands for the set
of maximal ideals). Moreover, it is enough to localize at primes $\fp$
or maximals $\fm$ that contain $I$ (considering that the improper
ideal $I=R$ has relation type $1$).

More generally, let $R$ be a Noetherian ring and let $U=\oplus _{q\geq
  0}U_{q}$ be a standard $R$-algebra, i.e., $U_0=R$ and $U$ is an
$R$-algebra finitely generated by elements $f_1,\ldots ,f_s$ of $U_1$.
Let $\sym(U_1)$ be the symmetric algebra of $U_1$ and
$\alpha:\sym(U_1)\to U$ be the canonical graded morhism induced by the
identity on $U_1$. For $q\geq 2$, define
$E(U)_q=\ker(\alpha_q)/U_1\cdot \ker(\alpha_{q-1})$. If
$\varphi:S=R[t_1,\ldots ,t_s]\to U$ is the graded polynomial
presentation of $U$ sending $t_i$ to $f_i$ and $L=\ker(\varphi)$, then
$E(U)_q\cong L_q/S_1L_{q-1}$, for all $q\geq 2$
(\cite[Theorem~2.4]{planas}). The relation type of $U$ can be defined
as $\rt(U)=\min\{ r\geq 1\mid E(U)_q=0\mbox{ for all }q\geq r+1\}$,
i.e., the maximum degree appearing in a minimal generating set of
equations of $U$. For $U=\rees(I)$ one has $E(\rees(I))_q=E(I)_q$ and
$\rt(\rees(I))=\rt(I)$, recovering the definitions above.

The double expression of $E(U)_q$ as $\ker(\alpha_q)/U_1\cdot
\ker(\alpha_{q-1})$ and $L_q/S_1L_{q-1}$ has an advantage; while the
first is canonical, the second is easier to deal with. To prove that
the relation type of an ideal of a polynomial ring solely depends on
the quotient ring, we use a third expression. Indeed, for each $q\geq
2$, there exists a graded isomorphism of $R$-modules $E(U)_q\cong
H_1(R,U,R)_q$, where $H_1(R,U,R)=\oplus_{q\geq 0}H_1(R,U,R)_q$ stands
for the homology of Andr\'e-Quillen (see \cite{andre} refering to the
homology of commutative rings and particularly
\cite[Remark~2.3]{planas} for this result).

\section{The invariance of the relation type}\label{invariance}

We begin by proving the invariance with respect to the quotient of the
relation type of an ideal in a polynomial ring.

\begin{theorem}\label{th-invariance}
Let $k$ be a field and $x_1,\ldots ,x_n$ and $y_1,\ldots ,y_m$
variables over $k$. Let $I$ be an ideal of $R=k[x_1,\ldots ,x_n]$
and let $J$ be an ideal of $S=k[y_1,\ldots ,y_m]$. Suppose that
there exists an isomorphism of $k$-algebras $R/I\cong S/J$. Then
$\rt(I)=\rt(J)$.
\end{theorem}
\begin{proof}
Let $\agr(I)=\rees(I)/I\rees(I)=\oplus_{q\geq 0}I^q/I^{q+1}$ be the
associated graded ring of $I$ and let $\agr(J)$ be the associated
graded ring of $J$.  Using \cite[Exercise~13, Chapter~V, \S~5]{kunz},
one deduces that there exists a graded isomorphism of $R/I$-algebras:
\begin{eqnarray*}
\agr(I)[y_1,\ldots ,y_m]\cong\agr(J)[x_1,\ldots ,x_n].
\end{eqnarray*}
Consider the natural augmentation morphisms. They induce the following
commutative diagram of homomorphisms of rings:
\begin{eqnarray*}
\xymatrix{ R/I \ar[r] \ar[d]^{\cong} & \agr(I)[y_1,\ldots ,y_m] \ar[r]
  \ar[d]^{\cong} & R/I \ar[d]^{\cong} \\ S/J \ar[r] &
  \agr(J)[x_1,\ldots ,x_n] \ar[r] & S/J.}
\end{eqnarray*}
This induces the graded isomorphism of homology groups:
\begin{eqnarray*}
H_1(R/I,\agr(I)[y_1,\ldots ,y_m],R/I)\cong H_1(S/J,\agr(J)[x_1,\ldots
  ,x_n],S/J).
\end{eqnarray*}
Applying the Jacobi-Zariski exact sequence of Andr\'e-Quillen homology
associated to the ring homomorphisms:
\begin{eqnarray*}
R/I\rightarrow \agr(I)\rightarrow \agr(I)[y_1,\ldots ,y_m]
\end{eqnarray*}
and the $\agr(I)[y_1,\ldots ,y_m]$-module $R/I$, for each $q\geq 2$,
one gets the isomorphism
\begin{eqnarray*}
E(\agr(I))_q\cong H_1(R/I,\agr(I),R/I)_q\cong
H_1(R/I,\agr(I)[y_1,\ldots ,y_m],R/I)_q.
\end{eqnarray*}
Analogously,
\begin{eqnarray*}
E(\agr(J))_q\cong H_1(S/J,\agr(J),S/J)_q\cong 
H_1(S/J,\agr(J)[x_1,\ldots, x_n],S/J)_q.
\end{eqnarray*}
Thus, for each $q\geq 2$, we have an isomorphism $E(\agr(I))_q\cong
E(\agr(J))_q$. In particular, $\rt(\agr(I))=\rt(\agr(J))$. It is known
that $\rt(\agr(I)))=\rt(\rees(I))$ (see \cite[Proposition~3.3]{planas}
or \cite[page 268]{hku}). Therefore the relation type of $\rees(I)$ is
equal to the relation type of $\rees(J)$, i.e., $\rt(I)=\rt(J)$.
\end{proof}

\begin{definition}
Let $A=R/I=k[x_1,\ldots ,x_n]/I$ be an affine $k$-algebra. We define
the {\em relation type} of $A$ as $\rt(A)=\rt(I)$, where $\rt(I)$ is
the relation type of the ideal $I$ of $R=k[x_1,\ldots ,x_n]$. If $V$
is an affine algebraic $k$-variety, the {\em relation type} of $V$ is
defined as $\rt(V)=\rt(k[V])$, the relation type of its coordinate
ring $k[V]$. By Theorem~\ref{th-invariance}, $\rt(A)$ and $\rt(V)$ are
well-defined.
\end{definition}

\begin{remark}
Let $I$ be an ideal of a Noetherian ring $R$ and let $J$ be an ideal
of a Noetherian ring $S$. Suppose that $R/I\cong S/J$. Then one cannot
deduce that $\rt(I)$ is equal to $\rt(J)$.
\end{remark}

\begin{example}
Consider the Neile's semicubical parabola $x^3-y^2=0$ in the complex
plane $\bc^2$ and its coordinate ring
$R=\bc[x,y]/(x^3-y^2)=\bc[\overline{x},\overline{y}]$, where
$\overline{x}$ and $\overline{y}$ stand for the classes of $x$ and $y$
in $R$. Let $\fm=(\overline{x}-1,\overline{y}-1)$ and
$\fn=(\overline{x},\overline{y})$ be the maximal ideals of $R$
corresponding to the regular point $(1,1)$ and to the origin,
respectively. Although the quotient rings $R/\fm$ and $R/\fn$ are
isomorphic (to $\bc$), $\rt(\fm)=1$ whereas $\rt(\fn)=2$. Note here
that $\fm=(x-1,y-1)/(x^3-y^2)$ and $\fn=(x,y)/(x^3-y^2)$, where
$(x-1,y-1)$ and $(x,y)$ are two ideals of the polynomial ring
$\bc[x,y]$ with $\bc[x,y]/(x-1,y-1)\cong \bc[x,y]/(x,y)$. In fact,
$\rt(x-1,y-1)=1$ and $\rt(x,y)=1$, since they are generated by a
regular sequence. However, $\fm$ and $\fn$ are ideals of $R$, which is
not a polynomial ring, and although $R/\fm\cong \bc[x,y]/(x-1,y-1)\cong
\bc[x,y]/(x,y)\cong R/\fn$, we cannot deduce that their relation type
coincide.
\end{example}
\begin{proof}
Since $x^3-y^2$ is irreducible, $R$ is a domain. From
$(x+y)(y-1)=(x^2+x+y)(x-1)-(x^3-y^2)$ one duduces that $\fm
R_{\mathfrak m}=(\overline{x}-1)R_{\mathfrak m}$ is locally generated
by a regular sequence, so $\rt(\fm)=1$ (see, e.g.,
\cite[Corollary~3.7]{hsv}). On the other hand, in $R$, we have the
strict inclusion of colon ideals
$(\overline{x}:\overline{y})\subsetneq \overline{x}\fn
:\overline{y}^2=R$. By \cite[Proposition~4.5]{planas}, $\rt(\fn)=2$.
\end{proof}

Following an idea of Eisenbud and Sturmfels in \cite[page~1]{es} and
as a consequence of Theorem~\ref{th-invariance}, we show that in order to
calculate the relation type of an ideal of a polynomial ring one can
suppose that the ideal is generated by trinomials, that is,
polynomials with at most three terms.

\begin{proposition}\label{trinomials}
Let $k$ be a field and $x_1,\ldots ,x_n$ variables over $k$. Let
$I=(f_1,\ldots ,f_s)$ be an ideal of $R=k[x_1,\ldots ,x_n]$. Then
there exist a polynomial ring $S=R[y_1,\ldots ,y_r]$ and a surjective
homomorphism of $R$-algebras $\sigma:S\to R$ such that
$\sigma^{-1}(I)$ is an ideal of $S$ generated by trinomials and such
that $\rt(I)=\rt(\sigma^{-1}(I))$.
\end{proposition}
\begin{proof}
Suppose that $f_1=g_1+\ldots +g_m$, where $g_i$ are monomial terms. If
$m\geq 4$, take $y_1$ a variable over $R$, set $R_1=R[y_1]$ and let
$\rho_1:R_1\to R$ be such that $\rho_1(y_1)=g_{m-1}+g_{m}$. Clearly
$\ker(\rho_1)=(y_1-(g_{m-1}+g_m))$. Let
$J_1=(y_1-(g_{m-1}+g_m),g_1+\ldots +g_{m-2}+y_1,f_2,\ldots
,f_s)$. Then $\rho_1(J_1)=I$ and, since $\ker(\rho_1)\subset J_1$,
$\rho_1^{-1}(I)=J_1$. In particular, $\rho_1$ induces an isomorphim of
$k$-algebras $R_1/J_1\cong R/I$ and, by Theorem~\ref{th-invariance},
$\rt(J_1)=\rt(I)$.

If $m=4$, $g_1+\ldots +g_{m-2}+y_1$ is already a trinomial. Suppose
that $m>4$. Recursively, for each $i=2,\ldots ,m-3$, take a new
variable $y_i$ over $R_{i-1}$, set $R_i=R_{i-1}[y_i]$ and let
$\rho_{i}:R_i\to R_{i-1}$ be such that
$\rho_i(y_i)=g_{m-i}+y_{i-1}$. Then
$\ker(\rho_i)=(y_i-(g_{m-i}+y_{i-1}))$. Let
\begin{eqnarray*}
J_i=(y_1-(g_{m-1}+g_m),\ldots ,y_i-(g_{m-i}+y_{i-1}),g_1+\ldots
+g_{m-i-1}+y_i, f_2,\ldots ,f_s).
\end{eqnarray*}
One has $\rho_i(J_i)=J_{i-1}$ and
$\rho_i^{-1}(J_{i-1})=J_i$. Therefore, $\rho_i$ induces an isomorphim
of $k$-algebras $R_i/J_{i}\cong R_{i-1}/J_{i-1}$ and
$\rt(J_i)=\rt(J_{i-1})$.  Set $S_1=R_{m-3}= R[y_1,\ldots ,y_{m-3}]$
and let $\sigma_1=\rho_{m-3}\circ\ldots\circ\rho_1$, where
$\sigma_1:S_1\to R$ is a surjective homomorphims of $R$-algebras. Set
$J=J_{m-3}$, where
\begin{eqnarray*}
J=J_{m-3}=(y_1-(g_{m-1}+g_m),\ldots
,y_{m-3}-(g_3+y_{m-2}),g_1+g_2+y_{m-3},f_2,\ldots ,f_s).
\end{eqnarray*}
Observe that $\sigma_1(J)=I$ and $\sigma^{-1}_1(I)=J$. Thus
$S_1/J=R_{m-3}/J_{m-3}\cong R_{m-4}/J_{m-4}$ and
$\rt(J)=\rt(J_{m-4})$, which is equal to $\rt(I)$. Note that in $J$ we
have replaced the polynomial $f_1$ by $m-2$ trinomials by introducing
$m-3$ new variables. To finish, proceed recursively with the rest of
the generators of $I$.
\end{proof}

\begin{remark}
An easy refinement of the argument above allows us to suppose that the
final ideal $\sigma^{-1}(I)$ is generated by polynomials of the
following kind: either monomials, or binomials with one of the two
terms being linear, or trinomials with all the three terms being
linear.
\end{remark}

\begin{remark}
From the proof of Proposition~\ref{trinomials} one obtains an
effective way to reduce a polynomial ideal to a trinomial ideal
preserving the relation type at the same time. However, its interest
seems more theoretical than practical due to its cost in introducing
more generators and more variables.
\end{remark}

\section{Equidimensional decomposition and relation 
type}\label{equidimensionalvsrt}

We start this section with some easy, but clarifying examples. From
now on, $R=k[x_1,\ldots ,x_n]$ will be a polynomial ring in $n$
variables $x_1,\ldots ,x_n$ over a field $k$, $I$ will be a proper
ideal of $R$ and $A=R/I$. Clearly if $R=k[x]$, then $R$ is a principal
ideal domain, and every proper ideal $I$ of $R$ is principal generated
by a nonzero divisor, hence $I$ is of linear type and $\rt(A)=1$. In
two variables we have the following simple example of a family of
algebras with unbounded relation type.

\begin{example}\label{rtIn=n}
Let $p\geq 1$ and $I=(x^p,y^p,x^{p-1}y)$ in $R=k[x,y]$. Set
$A=R/I=k[x,y]/I$. Then $\rt(I)=p$ and $\rt(A)=p$.
\end{example}
\begin{proof}
By \cite[Example~3.2]{planas}, we can localize at $\fm=(x,y)$ in order
to calculate the relation type of $I_p$. The result then follows from
\cite[Example~5.1]{mp}.
\end{proof}

However, the example above is not reduced. Allowing an arbitrary
number of variables, we give the following example of a family of
reduced algebras with unbounded relation type.

\begin{example}\label{cycle-even-length}
Let $p\geq 1$ and $n=2p$. Let $I$ be the monomial ideal of
$R=k[x_1,\ldots ,x_n]$ generated by the quadratic square free
monomials $x_1x_2,\ldots ,x_{n-1}x_n,x_nx_1$. Set $A=R/I$. Then
$\rt(I)=p$ and $\rt(A)=p$. Observe that $I$ is the edge ideal
$I(\mcg)$ associated to the graph $\mcg$, where $\mcg$ is a cycle of
even length $n=2p$.
\end{example}
\begin{proof}
The result follows from \cite[Section~3]{villarreal}.
\end{proof}

Before proceeding, we recall a central concept to our purposes in this
section. Given an irredundant primary decomposition of $I$, let
$I_{i,j}$ be the primary components of $I$ of a given height $i\geq
1$, for $j=1,\ldots ,r_i$. Set $I_i=I_{i,1}\cap\ldots\cap I_{i,r_{i}}$
and call it the $i$-th {\em equidimensional component of $I$}.  Note
that, for $i>\height(I)$, the $I_{i,j}$ (and hence $I_i$) are not
uniquely defined (see \cite[Definition~3.2.3]{vasconcelos-Comput}). To our
convenience, let us write $I_i=R$ if $I$ has no primary components of
height $i$. Henceford, $I$ can be expressed as $I=I_1\cap\ldots \cap
I_n$. Such a representation will be called an {\em equidimensional
  decomposition} of $I$ (associated to the given irredundant primary
decomposition of $I$). Note that, since the given primary
decomposition is irredundant, the equidimensional decomposition is
irredundant in the following sense: either $I_i=R$, or else $I_i$ is
an unmixed ideal of height $i$ such that $I_1\cap\ldots\cap
I_{i-1}\cap I_{i+1}\cap\ldots\cap I_n\not\subseteq I_i$. If $I$ is
radical, then each $I_i$ is either $R$, or else an unmixed radical
ideal of height $i$. With these notations, we start with the following
easy example.

\begin{example}\label{rt(I)=1}
Let $I$ be a proper radical ideal of $R=k[x,y]$. Then $\rt(I)=1$. In
particular, an affine $k$-algebra of embedding dimension at most $2$
has relation type $1$.
\end{example}
\begin{proof}
Let $I=I_1\cap I_2$ an equidimensional decomposition associated to an
irredundant primary decomposition of $I$, where either $I_1=R$, or
else $I_1=I_{1,1}\cap\ldots\cap I_{1,n_1}=(g_1)\cap\ldots\cap
(g_{n_1})=(g)$ is a principal ideal, with $g_j$ irreducible and
$g=g_1\cdots g_{n_1}$ (see \cite[Exercise~20.3]{matsumura}); moreover,
either $I_2=R$, or else $I_2=I_{2,1}\cap\ldots\cap I_{2,n_2}$, where
$I_{2,j}\in\Max(R)$, for $j=1,\ldots ,n_2$. If $I_1=R$, then
$I=I_2\neq R$. For each $\fm\in\Max(R)\setminus \{I_{2,1},\ldots
,I_{2,n_2}\}$, then $I_{\mathfrak{m}}=R_{\mathfrak{m}}$.  If
$\fm\in\{I_{2,1},\ldots ,I_{2,n_2}\}$, then
$I_{\mathfrak{m}}=\mathfrak{m}R_{\mathfrak{m}}$, which is generated by
a regular sequence. Thus
$\rt(I)=\sup\{\rt(I_{\mathfrak{m}})\mid\fm\in\Max(R)\}=1$. Suppose
that $I_1\neq R$. Then $I=(g)\cap I_2$. If $I_2=R$, then $I=(g)$ and
$\rt(I)=1$. Suppose that $I_2\neq R$. Then $g\not\in I_{2,j}$, due to
the irredundancy of the primary decomposition of $I$. Take $\fm$ a
maximal ideal of $R$ containing $I$. If $\fm$ contains $g$ (and hence
$\fm\neq I_{2,j}$ for all $j=1,\ldots ,n_2$), then
$I_{\mathfrak{m}}=(g)R_{\mathfrak{m}}$. If $\fm=I_{2,j}$, then
$I_{\mathfrak{m}}=\fm R_{\mathfrak{m}}$. So
$\rt(I)=\sup\{\rt(I_{\mathfrak{m}})\mid\fm\in\Max(R),\fm\supseteq
I\}=1$.
\end{proof}

The next result shows that, in embedding dimension $n\geq 3$ and
relation type at least $2$, the equidimensional components of
dimension $1$ and $n$ are irrelevant with respect to the relation
type.

\begin{proposition}\label{decomposition}
Let $I$ be a proper radical ideal of $R=k[x_1,\ldots ,x_n]$, $n\geq
3$. Let $I=I_1\cap\ldots\cap I_n$ be an equidimensional decomposition
of $I$. Then either $\rt(I)=1$, or $\rt(I)=\rt(I_2\cap\ldots\cap
I_{n-1})$, where $I_i\neq R$ for some $i=2,\ldots ,n-1$. (Both cases
may occur simultaneously.)
\end{proposition}
\begin{proof}
To simplify notations, set $L=I_2\cap\ldots\cap I_n$, so that
$I=I_1\cap L$. If $I_1=R$, then $L\neq R$ and $I=L$. Suppose that
$I_1\neq R$. Since $R$ is a unique factorisation domain, then
$I_1=(g)$ (see the proof of Example~\ref{rt(I)=1}). If $L=R$, then
$I=(g)$. Suppose that $L\neq R$. Write $I_i=I_{i,1}\cap\ldots\cap
I_{i,r_{i}}$ for all $i=2,\ldots ,n$, with $I_i\neq R$. Then $g\not\in
I_{i,j}$, for all $j=1,\ldots ,r_i$. Indeed, if $g\in I_{i,j}$, then
the given primary decomposition of $I$ would be redundant. Take now
$ag\in I=I_1\cap L=(g)\cap L$, with $a\in R$. Since $g\not\in
I_{i,j}$, for all for all $i=2,\ldots ,n$, with $I_i\neq R$, we deduce
that $a\in I_{i,j}$, because $I_{i,j}$ is prime. Therefore $a\in I_i$,
for all $i=2,\ldots ,n$, so that $ag\in g(I_2\cap\ldots\cap I_n)=gL$
and $I=gL$.

In conclusion, either $I=L$, or $I=(g)$, or $I=gL$, with $L\neq R$.
We know that $L$ and $gL$ have the same relation type (see, e.g., the
characterisation of the relation type in terms of the Andr\'e-Quillen
homology, \cite[Remark~2.3]{planas}). Therefore, either $\rt(I)=1$, or
$\rt(I)=\rt(L)$, where $L=I_2\cap\ldots\cap I_n\neq R$.

Suppose that $\rt(I)=\rt(L)$, with $L\neq R$. If $I_n=R$, then $L=
I_2\cap\ldots\cap I_{n-1}$ and $\rt(I)=\rt(I_2\cap\ldots\cap
I_{n-1})$, where $I_i\neq R$ for some $i=2,\ldots ,n-1$, and we are
done.

Suppose that $\rt(I)=\rt(L)$, with $L\neq R$, and that $I_n\neq
R$. Write $I_n=I_{n,1}\cap\ldots\cap I_{n,r_n}$, where $I_{n,j}$ are
maximal ideals of $R$. Set $J=I_2\cap\ldots\cap I_{n-1}$. Thus
$L=I_2\cap\ldots\cap I_{n-1}\cap I_n=J\cap I_n$. If $J=R$, then
$L=I_n$ and
$\rt(I)=\rt(L)=\rt(I_n)=\sup\{\rt((I_n)_{\mathfrak{m}})\mid
\fm\in\{I_{n,1},\ldots ,I_{n,r_n}\}\}=1$, because in this case,
$(I_n)_{\mathfrak{m}}=\fm R_{\mathfrak{m}}$, which is generated by a
regular sequence. Suppose that $J\neq R$.  For each
$\fm\in\Max(R)\setminus\{I_{n,1}\ldots ,I_{n,r_n}\}$,
$L_{\mathfrak{m}}=J_{\mathfrak{m}}\cap
(I_n)_{\mathfrak{m}}=J_{\mathfrak{m}}$, because
$(I_n)_{\mathfrak{m}}=R_{\mathfrak{m}}$. If $\fm\in\{I_{n,1},\ldots
,I_{n,r_n}\}$, then $J\not\subseteq \fm$, so
$J_{\mathfrak{m}}=R_{\mathfrak{m}}$,
$L_{\mathfrak{m}}=J_{\mathfrak{m}}\cap (I_{n})_{\mathfrak{m}}=\fm
R_{\mathfrak{m}}$ and $\rt(L_{\mathfrak{m}})=1$.  Thus,
\begin{eqnarray*}
&&\rt(L)=\sup\{\rt(L_{\mathfrak{m}})\mid\fm\in\Max(R)\}=
  \sup\{\rt(L_{\mathfrak{m}})\mid\fm\in\Max(R)\setminus\{I_{n,1}\ldots
  ,I_{n,r_n}\}\}=\\&& \sup\{\rt(J_{\mathfrak{m}})\mid
  \fm\in\Max(R)\setminus\{I_{n,1}\ldots
  ,I_{n,r_n}\}\}=\sup\{\rt(J_{\mathfrak{m}})\mid
  \fm\in\Max(R)\}=\rt(J).
\end{eqnarray*}
Therefore, $\rt(I)=\rt(L)=\rt(J)=\rt(I_2\cap\ldots\cap I_{n-1})$,
where $I_i\neq R$ for some $i=2,\ldots ,n-1$.
\end{proof}

\begin{proposition}\label{1or3}
Let $I$ be a proper radical ideal of $R=k[x,y,z]$. Then either
$\rt(I)=1$, or $\rt(I)=\rt(I_2)$, where $I_2\neq R$, the second
equidimensional component of $I$, is syzygetic. Moreover, either
$\rt(I)=1$, or else $\rt(I)\geq 3$.
\end{proposition}
\begin{proof}
By Proposition~\ref{decomposition}, we can suppose that
$\rt(I)=\rt(I_2)$, where $I_2=I_{2,1}\cap\ldots\cap I_{2,n_2}\neq R$,
with $I_{2,j}$ prime ideals of $R$ of height $2$. In particular, $I_2$
is generically a complete intersection (i.e., a complete intersection
localized at each minimal prime of $I$), and a perfect ideal of
projective dimension $1$. Hence $I_2$ is syzygetic, i.e.,
$E(I_2)_2=\ker(\alpha_2:\sym_2(I_2)\to I_2^2)=0$ (see the subsequent
Remark to \cite[Proposition~2.7]{hsv}).

If $I_2$ is a complete intersection or an almost complete
intersection, then $I_2$ is of linear type (see, e.g.,
\cite[Theorem~4.8]{hmv}). Thus, $\rt(I)=\rt(I_2)=1$.

If $I_2$ is generated by at least fours elements, then $I_2$ is not of
linear type (see \cite[Proposition~2.4]{hsv}; see also
\cite[Theorem~5.1]{tchernev}). Thus $\rt(I_2)\geq 2$. Moreover,
$\rt(I_2)>2$. Indeed, if $\rt(I_2)\leq 2$, since $\rt(I_2)=\min\{
r\geq 1\mid E(I_2)_q=0\mbox{ for all }q\geq r+1\}$, then $E(I_2)_q=0$
for all $q\geq 3$. Since $E(I_2)_2=0$, then it would follow that
$\rt(I_2)=1$, a contradiction. Therefore $\rt(I)=\rt(I_2)\geq
3$. (Note that we do not affirm that $I$ is syzygetic.)
\end{proof}

As an immediate consequence we have the following result.

\begin{corollary}\label{nonexistence}
There do not exist affine $k$-algebras of embedding dimension $3$ and
relation type $2$. There do not exist affine algebraic $k$-varieties
of embedding dimension $3$ and relation type $2$.
\end{corollary}

\section{Examples in embedding dimension three}\label{examples}

In this section we focus our attention on affine $k$-algebras of
embedding dimension three. Our purpose is to give illustrative
examples of affine algebras with different relation types. So now,
$R=k[x,y,z]$ will be a polynomial ring in three variables $x,y,z$ over
a field $k$, $I$ will be a proper radical ideal of $R$ and
$A=R/I$. According to Proposition~\ref{1or3}, we can suppose that $I$
is an irredundant intersection of prime ideals of height $2$. In
particular, $I$ is a perfect ideal of projective dimension
$1$. Suppose that $I=(f_1,\ldots ,f_s)$ is minimally generated by
$s\geq 2$ elements and that
\begin{eqnarray*}
0\longrightarrow R^{s-1}\buildrel\eta\over\longrightarrow
R^{s}\longrightarrow I\longrightarrow 0
\end{eqnarray*}
is a presentation of $I$. By the Theorem of Hilbert-Burch, there
exists an element $g\in R$, $g\neq 0$, such that $I=gI_{s-1}(\eta)$,
where $I_{s-1}(\eta)$ is the determinantal ideal generated by the
$(s-1)\times (s-1)$ minors of the $s\times (s-1)$ matrix $\eta$ (see,
e.g., \cite[Theorem~1.4.16]{bh}). In particular,
$\rt(I)=\rt(I_{s-1}(\eta))$. Therefore, in terms of the computation of
the relation type, we can directly suppose that $I=I_{s-1}(\eta)$.

On taking the symmetric functor in the short
exact sequence above one gets:
\begin{eqnarray*}
0\longrightarrow L\langle 1\rangle \longrightarrow
S=\sym(R^s)=R[t_1,\ldots ,t_s]\xrightarrow{\Phi}
\sym(I)\longrightarrow 0,
\end{eqnarray*}
where $L\langle 1\rangle$ is the defining ideal of $\sym(I)$. Recall
the notation in Section~\ref{notations}, where $L=\ker(\varphi)$ is
the ideal of equations of $I$ given by the polynomial presentation
$\varphi:\alpha\circ\Phi:S\to \rees(I)$ and
$\alpha:\sym(I)\to\rees(I)$ is the canonical morphism. If we denote by
$[t_1,\ldots ,t_s]$ the $1\times s$ matrix of entries $t_i$, then
$L\langle 1\rangle=(g_1,\ldots,g_s)S$, with $[g_1\ldots
  ,g_s]=[t_1,\ldots ,t_s]\cdot\eta$. One can write this last
expression as
\begin{eqnarray*}
[g_1\ldots ,g_s]=[t_1,\ldots ,t_s]\cdot\eta =[x,y,z]\cdot B(\eta),
\end{eqnarray*}
where $B(\eta)$ is a $3\times s$ matrix of linear forms in the
variables $t_i$. This matrix is called a {\em Jacobian dual} matrix of
$\eta$. One can show that the determinantal ideal $I_3(B(\eta))$
generated by the $3\times 3$ minors of $B(\eta)$, is included in $L$,
the ideal of equations of $I$ (see
\cite[Proposition~7.2.3]{vasconcelos-Comput}).

The ideal $I$ is said to have the {\em expected equations} if
$L=(L\langle 1\rangle, I_3(B(\eta)))$ (see
\cite[Definition~1.8]{vasconcelos-Integral}; see \cite[Introduction
  and Section~3.1]{vasconcelos-Crelle}, where this ideas were first
stated; see also \cite{morey}, \cite{mu}). Most of the results on the
expected equations of an ideal $I$ of a ring $R$, suppose that either
$R$ is a Noetherian local ring, or else $R$ is a standard graded and
$I$ is a homogeneous ideal.

Although in our case $I$ is not homogeneous, we can use these
techniques as a first approach to obtain the equations of $I$ and its
relation type, as the next example shows.

\begin{example}\label{int2herzog}
Let $n=(n_1,n_2,n_3)$, $m=(m_1,m_2,m_3)\in\bn^3$, with $\gcd(n)=1$ and
$\gcd(m)=1$. Let $\fp_n$ be the kernel of the $k$-homomorphism
$R=k[x,y,z]\to k[t]$ which sends $x$, $y$ and $z$, to $t^{n_1}$,
$t^{n_2}$ and $t^{n_3}$, respectively. Similarly, one defines
$\fp_m$. Clearly, $\fp_n$ and $\fp_m$ are prime. It is known that they
are either a complete intersection, or else an almost complete
intersection, and in particular, of relation type $1$ (see, e.g.,
\cite[Example~V.3.13,~$f)$]{kunz} and \cite[Theorem~4.8]{hmv}). Fix
now $n=(3,4,5)$ and $m=(3,4,3r)$, for some $r\geq 3$ and let
$I=\fp_n\cap\fp_m$. Using {\sc Singular} \cite{singular}, one gets the
minimal system of generators for $I$:
\begin{eqnarray*}
&&f_1=x^4-y^3\mbox{ , }f_2=-x^{r+1}z+x^ry^2+xz^2-y^2z\mbox{ , }\\&&
f_3=-x^{r+3}+x^ryz+x^3z-yz^2\mbox{ and }f_4=-x^{r+2}y+x^rz^2+x^2yz-z^3,
\end{eqnarray*}
and the presentation $0\longrightarrow
R^3\buildrel\eta\over\longrightarrow R^4\longrightarrow
I\longrightarrow 0$, with
\begin{eqnarray*}
\eta=\left(\begin{array}{ccc}
  \phantom{m}z-x^r&0&0\\-y&z&-x^2\phantom{m}\\-x&-y\phantom{m}&z
  \\\phantom{m}0&x&-y\phantom{m}
\end{array}\right).
\end{eqnarray*}
In other words, $I$ is generated by the $3\times 3$ minors of the
$4\times 3$ matrix $\eta$. Thus $L\langle 1\rangle$ is generated by
$g_1,g_2,g_3$, where $[g_1,g_2,g_3]=[t_1,\ldots ,t_4]\cdot
\eta=[x,y,z]\cdot B(\eta)$, and $B(\eta)$ is the Jacobian dual of
$\eta$. Concretely,
\begin{eqnarray*}
B(\eta)=\left(\begin{array}{ccc}
-t_3-x^{r-1}t_1&t_4&-xt_2\\ 
-t_2&-t_3\phantom{m}&-t_4\\ 
\phantom{m}t_1&t_2&\phantom{n}t_3
\end{array}\right).
\end{eqnarray*}
Using Singular \cite{singular} again, one deduces that
$L=\ker(\varphi)$ is generated by $g_1,g_2,g_3$ and $\det(B)$. In
particular, $\rt(I)=3$. Note that, as in
\cite[Theorem~3.1.1]{vasconcelos-Crelle}, $I$ has the {\em expected
  equations}, three of degree $1$ and exactly one of degree $3$,
though in the aforementioned result, $R$ is supposed to be a regular
local ring and $I$ is supposed to be a prime ideal, whereas in our
case, $R=k[x,y,z]$ and $I$ is not prime nor homogeneous.
\end{example}

The next example shows that all the cases arising in the proof of
Proposition~\ref{1or3} can occur.

\begin{example}\label{int2HN}
Let us consider the intersection of two ideals of Herzog-Northcott
type. Recall that an ideal of Herzog-Northcott type of $R=k[x,y,z]$ is
the determinantal ideal $J$ generated by the $2\times 2$ minors of a
$2\times 3$ matrix with rows $x^{a_1},y^{a_2},z^{a_3}$ and
$y^{b_{2}},z^{b_3},x^{b_{1}}$, where $(a_1,a_2,a_3)$ and
$(b_1,b_2,b_3)\in\bn^3$, and $\bn$ is the set of positive integers.
In other words,
$J=(x^{c_1}-y^{b_2}z^{a_3},y^{c_2}-x^{a_1}z^{b_3},z^{c_3}-x^{b_1}y^{a_2})$,
where $c_i=a_i+b_i$, for $i=1,2,3$.  It is known that such an ideal
$J$ is of linear type and, if $k$ has characteristic zero (or large
enough), then $J$ is radical (see \cite[Remark~6.4 and
  Theorem~9.1]{hn}). Moreover, $J$ is prime if and only if
$\gcd(m(J))=1$, where
$m(J)=(c_2c_3-a_2b_3,c_1c_3-a_3b_1,c_1c_2-a_1b_2)$ (see
\cite[Definition~7.1, Remark~7.2 and Theorem~7.8]{hn}).

Consider now the following four ideals of Herzog-Northcott type:
$J_1=(x^3-yz,y^3-x^2z,z^2-xy^2)$, $J_2=(x^3-yz,y^2-xz,z^2-x^2y)$,
$J_3=(x^2-y^2z,y^3-xz,z^2-xy)$ and $J_4=(x^2-yz,y^2-xz^2,z^3-xy)$.
Observe that $m(J_1)=(4,5,7)$, $m(J_2)=(3,4,5)$, $m(J_3)=(5,3,4)$ and
$m(J_4)=(4,5,3)$. In particular, $J_i$ are prime ideals and $J_i\cap
J_j$ are Cohen-Macaulay ideals of projective dimension $1$, for
$i,j=1,2,3,4$.  {\sc Singular} (\cite{singular}) shows that $J_1\cap
J_2$ is a complete intersection and that $J_1\cap J_4$ is an almost
complete intersection, in particular, $J_1\cap J_2$ and $J_1\cap J_4$
are ideals of linear type. However, $J_1\cap J_3$ is minimally
generated by four elements. Therefore $J_1\cap J_3$ is not of linear
type type (see \cite[Proposition~2.4]{hsv}). By
Proposition~\ref{1or3}, $J_1\cap J_3$ is syzygetic and has relation
type at least $3$.
\end{example}

\section{The relation type of irreducible affine space 
curves}\label{spacecurves}

In this section, we give some examples where $I=\fp$ is a prime ideal
of height $2$ in $R=k[x,y,z]$. Alternatively, these are examples of
irreducible affine algebraic curves $V$ in the three dimensional
affine $k$-space $\ba^3(k)$. We start with monomial curves.

\begin{example}\label{e-herzog}
Let $m=(m_1,m_2,m_3)\in\bn^3$ with $\gcd(m_1,m_2,m_3)=1$. Let
$V\subset \ba^3(k)$ be the parametrized curve
$V=\{(\lambda^{m_1},\lambda^{m_2},\lambda^{m_3})
\in\ba^3(k)\mid\lambda\in k\}$. Suppose that $k$ is infinite. Then
$\rt(V)=1$.
\end{example}
\begin{proof}
Let $\fp\subset R=k[x,y,z]$ be the kernel of the $k$-homomorphism
$R=k[x,y,z]\to k[t]$ sending $x$, $y$, $z$ to $t^{m_1}$, $t^{m_2}$,
$t^{m_3}$, respectively. According to \cite[Lemma~3.4]{ev}, if
$\gcd(m_1,m_2,m_3)=1$, then $V= \fV(\fp)$ is the affine algebraic
$k$-variety defined by the ideal $\fp\subset R=k[x,y,z]$. Moreover if
$k$ is infinite, the vanishing ideal of $V$ is $\fI(V)=\fp$ (see,
e.g., \cite[Corollary~7.1.12]{villarreal-book}). Thus
$k[V]=k[x,y,z]/\fI(V)=k[x,y,z]/\fp$, where $\fp$ is either a complete
intersection ideal, or else an almost complete intersection, hence
$\fp$ is of linear type (\cite[Example~V.3.13,~$f)$]{kunz} and
\cite[Theorem~4.8]{hmv}).  In particular, $\rt(\fp)=1$ and $\rt(V)=1$.
\end{proof}

The following example is taken from
\cite[Section~3.1]{vasconcelos-Crelle}. It also shows that the
relation type may depend on the characteristic of the ground field.

\begin{example}\label{e-vasconcelos}
Let $\fp\subset R=k[x,y,z]$ be the kernel of the $k$-homomorphism
$R=k[x,y,z]\to k[t]$ which sends $x$, $y$ and $z$ to $t^{6}$, $t^{8}$
and $t^{10}+t^{11}$, respectively.  Let $V=\fV(\fp)\subset \ba^3(k)$
be the affine algebraic curve defined by $\fp$. If $k$ is infinite,
then $\rt(V)=3$. If $k=\bz/2\bz$, $\rt(V)=1$.
\end{example}
\begin{proof}
Let $W=\{(\lambda^6,\lambda^8,\lambda^{10}+\lambda^{11})\mid
\lambda\in k\}$. For all $f\in\fp$,
$f(\lambda^6,\lambda^8,\lambda^{10}+\lambda^{11})=\psi(f)(\lambda)=0$
and $W\subseteq V=\fV(\fp)$.  In particular, $\fp\subseteq
\fI(\fV(\fp))=\fI(V) \subseteq \fI(W)$. Let $f\in\fI(W)$, i.e.,
$f(\lambda^6,\lambda^8,\lambda^{10}+\lambda^{11})=0$, for all
$\lambda\in k$. Set $g(t)=f(t^6,t^8,t^{10}+t^{11})=\psi(f)$. Then
$g(\lambda)=0$, for all $\lambda\in k$. Suppose that $k$ is
infinite. Then $g(t)=0$ and $f\in\fp$. Thus $\fI(W)=\fp$, so
$\fI(V)=\fp$. Using {\sc Singular}, one deduces that $\fp$ is
minimally generated by four elements and that $\fp$ has the expected
equations (see also
\cite[Proposition~3.1.1]{vasconcelos-Crelle}). Thus $\rt(\fp)=3$ and
$\rt(V)=3$. Suppose that $k=\bz/2\bz$. Clearly $(0,0,0)$ and $(1,1,0)$
are in $W\subseteq V=\fV(\fp)$.  Moreover $f_1=x^4+y^3$ and
$f_2=x^2y+xy^2+z^2$, say, are easily seen to be in $\fp$, so in
$\fI(V)$. Since $f_1$ does not vanish on $(0,1,0)$, $(0,1,1)$,
$(1,0,0)$ and $(1,0,1)$ and $f_2$ does not vanish on $(0,0,1)$ and
$(1,1,1)$, it follows that these six points are not in $V$. Hence
$W=V=\{(0,0,0),(1,1,0)\}$ and $\fI(W)=\fI(V)$. Clearly
$x+y,z\in\fI(V)\setminus \fp$ and $\fp\subsetneq \fI(V)$. In fact,
$\fI(V)=(x+y,z)$, which is generated by a regular sequence, so
$\fI(V)$ is of linear type.
\end{proof}

The next curve has relation type $3$ too (if $k$ is
infinite). Moreover, when $k=\bc$, it is known to be a
set-theoretically complete intersection (see
\cite[Example~3.7]{huneke}).

\begin{example}\label{e-huneke}
Let $\fp\subset R=k[x,y,z]$ be the kernel of the $k$-homomorphism
$R=k[x,y,z]\to k[t]$ which sends $x$, $y$ and $z$ to $t^{6}$,
$t^{7}+t^{10}$ and $t^{8}$, respectively.  Let $V=\fV(\fp)\subset
\ba^3(k)$ be the affine algebraic curve defined by $\fp$.  If $k$ is infinite,
then $\rt(V)=3$. If $k=\bz/2\bz$, $\rt(V)=1$.
\end{example}
\begin{proof}
The proof is analogous to the former one. To deduce the equations of
$\fI(V)$ one can use {\sc Singular} (see \cite[Proposition~3.1.1 and
  Example~3.2.1]{vasconcelos-Crelle}) and
\cite[Example~3.7]{huneke}).
\end{proof}

\begin{remark}
In the examples above (\ref{e-herzog}, \ref{e-vasconcelos} or
\ref{e-huneke}) we have obtained irreducible affine algebraic space
curves with relation type $1$ and $3$. What are the geometric
properties that make them have different relation type? What is the
geometric meaning of the relation type of an irreducible affine
algebraic space curve?
\end{remark}

\begin{closingremark}
Although the examples above do not exceed four generators, they have
costly computations. To obtain a general procedure to find the
equations of prime ideals minimally generated by an arbitrary number
of elements seems a very difficult task. For instance, one could ask
for the equations, or just the relation type, of the prime ideals
defined by Moh in \cite{moh}, a question to which we do not have an
answer at the present moment. We intend to pursue this problem in
future work. 
\end{closingremark}

\section*{Acknowledgment} 
The question whether the relation type of an ideal is an invariant of
the quotient ring was raised to me by Jos\'e M. Giral. I thank him for
his generosity and for sharing his ideas and knowledge. I also want to
thank Philippe Gimenez, Liam O'Carroll, Bernd Ulrich, Wolmer
V. Vasconcelos and Santiago Zarzuela for stimulating comments and
suggestions.


\vspace{0.2cm}
{\footnotesize 

\noindent {\sc Departament de Matem\`atica Aplicada~1, Universitat
  Polit\`ecnica de Catalunya}, \\ Diagonal 647, ETSEIB, 08028 Barcelona,
Catalunya. {\em E-mail address}: francesc.planas@upc.edu }
\end{document}